\documentclass[a4paper]{article}

\usepackage{amssymb}
\usepackage{amsthm}
\usepackage{amsmath}
\usepackage{siunitx}
\usepackage{hyperref}
\usepackage{cleveref}
\usepackage[utf8]{inputenc}
\usepackage{csquotes}
\usepackage{comment}
\usepackage{booktabs}
\usepackage{longtable}
\usepackage{adjustbox}
\usepackage{array}
\usepackage{url}
\usepackage{titlesec}
\usepackage{authblk}
\usepackage{xcolor} % Load the xcolor package for color options

\newcommand{\Hi}{\mathbb{H}^2}
\newcommand{\HitR}{\mathbb{H}^2 \times \mathbb{R}}
\newcommand{\Sigmaw}{\mathcal W}

\newtheorem{theorem}{Theorem}
\newtheorem{lemma}[theorem]{Lemma}
\newtheorem{claim}[theorem]{Claim}
\newtheorem{corollary}[theorem]{Corollary}
\newtheorem{definition}[theorem]{Definition}
\newtheorem{proposition}[theorem]{Proposition}

\titleformat{\subsection}
  {\mdseries\itshape\large} % Medium series, italic shape, and large font size
  {\thesubsection}{1em}{} % Numbering, spacing, and the section title itself

\author[1]{Giuseppe Pipoli}
\author[2]{Jo\~{a}o Paulo dos Santos}
\author[3]{Giuseppe Tinaglia}

\affil[1]{Department of Information Engineering, Computer Science and Mathematics,
Università degli Studi dell’Aquila}
\affil[2]{Departamento de Matemática, Universidade de Brasília, Brazil}
\affil[3]{King's College London, Mathematics, London, U.K.}

\title{On the geometry of the asymptotic boundary of translators in $\mathbb H^2\times \mathbb R$}

\begin{document}
\maketitle

\begin{abstract}
In this work, we study complete properly immersed translators in the product space
$\HitR$, focusing on their asymptotic behavior at infinity. We classify the asymptotic boundary components of these translators under suitable continuity assumptions. Specifically, we prove that if a boundary component lies in the vertical asymptotic boundary, it is of the form $\{p\}\times [T,\infty)$ or $\{p\}\times \mathbb R$, while if it lies in the horizontal asymptotic boundary, it is a complete geodesic. Our approach is inspired by earlier work on minimal and constant mean curvature surfaces in $\HitR$, with a key ingredient being the use of  symmetric translators  as barriers.
\end{abstract}

\section{Introduction}
\label{sec:introduction}

Translators arise naturally in the study of mean curvature flow as self-similar solutions that move by translation. Their geometric and analytic properties have been extensively studied, particularly in Euclidean space, see for instance \cite[Chapter 13]{book-EGF} , mainly because of their role in understanding of type-II singularities of mean curvature flow \cite{HS}. More recently, translators have been studied in hyperbolic settings such as $\mathbb H^3$, see for instance \cite{Bueno-Lopez-H3-1, Bueno-Lopez-H3-2, deLima-Ramos-dosSantos-H3}, and $\HitR$, see for instance~\cite{bueno-2018, bueno-lopez-2025,  helicoidal,lima-pipoli,GHLM-serrin, lira-martin-2019}. In this work, we focus on complete properly immersed translators in the product space 
$\HitR$, analyzing their asymptotic behavior at infinity.

$\HitR$ is one of Thurston’s eight model geometries, making it a natural setting for studying curvature-driven flows beyond the Euclidean case. The interplay between the negatively curved base and the product structure introduces new rigidity phenomena and influences the qualitative behavior of solutions. Understanding translators in this setting provides insight into the long-time behavior of mean curvature flow in negatively curved spaces.

We establish a classification of the asymptotic boundary components of such translators under suitable continuity assumptions and state here a representative special case.

\begin{theorem}\label{main}
Let $M$ be a complete properly immersed translator in $\HitR$ whose asymptotic boundary, $\partial_\infty M$, is  a proper family of disjoint immersed continuous curves so that $M\cup \partial_\infty M $ is a continuous surface with boundary.  Then, the following holds:
\begin{itemize}
    \item if $\gamma$ is a connected component of its vertical asymptotic boundary, then $\gamma=\{p\}\times [T,\infty)$ with $p\in \partial_\infty \Hi$ and $T\in
    \mathbb{R}$ or $\gamma=\{p\}\times \mathbb{R}$.
    \item if $\gamma$ is a connected component of its horizontal asymptotic boundary, then $\gamma$ is a complete geodesic.
\end{itemize}
\end{theorem}
For example, see Corollary~\ref{nonfinitethm} for a sharper result concerning the horizontal asymptotic boundary. In fact, the second bullet in Theorem~\ref{main} is a consequence of the following theorem.

\begin{theorem}\label{main2}
Let $M$ be a complete properly immersed translator in $\HitR$. Let $\alpha^\pm\subset\partial_{\infty}^{\pm} M\subset \Hi\times {\{\pm \infty\}}$ be a properly embedded arc in the vertical boundary of $M$ such that $\overline{[\partial_{\infty}^{\pm} M\setminus \alpha^\pm]}\cap \alpha^\pm=\emptyset$ and  $M\cup \alpha^\pm $ is a continuous surface with boundary.  Then, $\alpha^\pm$ is a geodesic arc in $\Hi\times \{\pm\infty\}$.
\end{theorem}

  Our approach is inspired by techniques developed to study the  geometry of the asymptotic boundary of minimal and constant mean curvature surfaces in 
$\HitR$ \cite{klma, nsest1,et4, et5}, with a key ingredient being the use of  symmetric translators  as barriers \cite{bueno-lopez-2025,lima-pipoli, lira-martin-2019}.  Throughout this paper, we adopt the cylinder model for 
 $\HitR$. This paper is organized as follows. In Section~\ref{geocomp}, we review the geodesic compactification of $\HitR$ and define the different components of the asymptotic boundary of a subset of $\HitR$.
 In Section~\ref{MCF}, we give a brief introduction to the study of translators of mean curvature flow in
$\HitR$, focusing on properties that are needed for the proof of our main results. 
In Section~\ref{symmetric-translators}, we describe examples of symmetric translators, focusing on the examples needed for the proof of our main results. 
Finally, we discuss the geometry of the vertical (resp.  horizontal) asymptotic boundary of translators of $\HitR$ in Section \ref{finite} (resp. Section \ref{nonfinite}).

\section{Geodesic Compactification of \texorpdfstring{$\mathbb{H}^2 \times\mathbb{R}$}{H2 x R}}\label{geocomp}

In this section, we recall the construction of the geodesic compactification of $\HitR$, see for instance~\cite{et5}. We first define convergence of sequences in 
$\Hi$
  to points at infinity. Then, we establish the structure of the asymptotic boundary of
$\HitR$, distinguishing between finite and non-finite asymptotic points.

 Let $y_0 \in \Hi$ be a fixed point of $\Hi$ and let $x_{\infty} \in \partial_{\infty}\Hi$. We denote by $[y_0, x_{\infty}) \subset \Hi$ the geodesic ray issuing from $y_0$ and with asymptotic boundary $x_{\infty}$. For any $\rho >0$ we denote by $\gamma_{\rho} \subset \Hi$ the geodesic intersecting the ray $[y_0, x_{\infty})$ orthogonally at point $y_{\rho}$ such that $d_{\Hi} (y_0, y_{\rho}) = \rho$. Let $\gamma^+_{\rho}$ be the component of $\Hi \setminus \gamma_{\rho}$ that contains $x_{\infty}$ in its asymptotic boundary.
 
Let $(x_n)_{n\in\mathbb N}$ be a sequence of points in $\Hi$. Using the previously defined notation, we say that $(x_n)_{n\in\mathbb N}$ converges to $x_{\infty}$, denoted by $x_n \to x_{\infty}$, if for any $\rho > 0$ there exists $n_{\rho} \in \mathbb{N}$ such that $x_n  \in \gamma^+_{\rho}$ for any $n \geq n_{\rho}$.

We observe that if we choose the Poincaré disc model of $\Hi$, then $x_n \to x_{\infty}$ if and only if the sequence $(x_n)_{n\in\mathbb N}$ converges to $x_{\infty}$ in the Euclidean sense.

We next discuss the asymptotic boundary of $ \mathbb{H}^2 \times \mathbb{R} $. First note that
    \[
    \partial_\infty(\mathbb{H}^2 \times \mathbb{R}) := [\partial_\infty \mathbb{H}^2 \times \mathbb{R}] \cup [\mathbb{H}^2 \times \{+\infty\}]\cup  [\mathbb{H}^2 \times \{-\infty\}] \cup [\partial_\infty \mathbb{H}^2 \times \{ +\infty\}]\cup [\partial_\infty \mathbb{H}^2 \times \{-\infty \}].
    \]
    This decomposition means that for a divergent sequence $(p_n)_{n\in\mathbb N}$ of $\mathbb{H}^2 \times \mathbb{R}$, there are three possibilities for converging to infinity (up to extracting a subsequence). That is, setting $p_n = (x_n, t_n) \in \mathbb{H}^2 \times \mathbb{R}$, we have the following cases:
    \begin{itemize}
        \item $x_n \to x_\infty \in \partial_\infty \mathbb{H}^2$ and $t_n \to t_0 \in \mathbb{R}$. We say that $p_\infty := (x_\infty, t_0) \in \partial_\infty \mathbb{H}^2 \times \mathbb{R}$ is an asymptotic point at finite height.
        \item $x_n \to x_0 \in \mathbb{H}^2$ and $t_n \to \pm \infty$. That is, $(p_n)_{n\in\mathbb N}$ converges to $p_\infty := (x_0, \pm \infty) \in \mathbb{H}^2 \times \{-\infty, +\infty\}$.
        \item $x_n \to x_\infty \in \partial_\infty \mathbb{H}^2$ and $t_n \to \pm \infty$. That is, $(p_n)_{n\in\mathbb N}$ converges to $p_\infty := (x_\infty, \pm \infty) \in \partial_\infty \mathbb{H}^2 \times \{-\infty, +\infty\}$.
    \end{itemize}

  Let $A \subset \mathbb{H}^2 \times \mathbb{R}$ be a nonempty subset. We say that a point $p_\infty \in \partial_\infty(\mathbb{H}^2 \times \mathbb{R})$ is an asymptotic point of $A$ if there is a sequence $(p_n)_{n\in\mathbb N}$ of $A$ converging to $p_\infty$. The set of asymptotic points of $A$, called the asymptotic boundary of $A$, is denoted by $\partial_\infty A$.
    
The set of asymptotic points at finite height is called the {\bf vertical asymptotic boundary} and is denoted by $\partial_\infty^v A$ (Note that in \cite{et5} the name \emph{finite asymptotic boundary} is used). The complement $\partial_\infty A \setminus \partial_\infty^v A$ is called the {\bf non-finite asymptotic boundary} of $A$.  With regards to the non-finite asymptotic boundary of $A$, we distinguish between the points that are contained in $\Hi \times \left\{ \pm \infty \right\}$ and the points that are contained in $\partial_\infty \mathbb{H}^2 \times \{ \pm \infty\}$. We call the non-finite asymptotic boundary contained in $\Hi \times \left\{ \pm \infty \right\}$ the {\bf horizontal asymptotic boundary}---respectively positive or negative---and we denote it by $\partial_{\infty}^{\pm}A$ respectively. We call the non-finite asymptotic boundary contained in $\partial_\infty \mathbb{H}^2 \times \{ \pm \infty\}$ the {\bf corner asymptotic boundary}---respectively positive or negative---and we denote it by $\partial^{c,\pm}_{\infty}A$ respectively.

\section{Translators of the Mean Curvature Flow}\label{MCF}
An immersion $f: M^n \rightarrow \widetilde{M}^{n+1}$ of a smooth manifold $M^n$ into a Riemannian ambient space $\widetilde{M}^{n+1}$ evolves under the mean curvature flow (MCF for short) if there exists a smooth one-parameter family of immersions $F : M \times I \rightarrow \widetilde{M}$, $t \in [0,T)$ and $F(\cdot,0)=f(\cdot)$ such that
\begin{equation}
\dfrac{\partial F}{\partial t} (p,t) = \vec{H}(p,t),\, (p,t) \in M \times I, \label{eq:main-MCF}
\end{equation}
where $\vec{H}(p,t)$ is the mean curvature vector field of the immersion $F_t : M \rightarrow \widetilde{M}$, where $F_t(p):=F(p,t)$ for a fixed $t \in I$.

A soliton solution to the MCF is given by $F(p,t)={\Gamma}_t(f(p))$ where ${\Gamma}_t$ is a one-parameter subgroup of the group of isometries of $\widetilde{M}$. Denoting with $\xi \in T\widetilde{M}$ the corresponding Killing vector field, it is well known (see for example \cite{HS-2000}) that $f: M^n \rightarrow \widetilde{M}^{n+1}$ evolves under a soliton solution to MCF corresponding to $\Gamma_t$ if and only if the mean curvature $H$ of $f$ and its unit normal $N$ satisfies, up to tangential diffeomorphisms,
\begin{equation}
H = \langle N, \xi \rangle. \label{eq:soliton-general}
\end{equation}

An important class of soliton solutions to the MCF is given by the \emph{translating solitons} in the Euclidean space $\mathbb{R}^{n+1}$, i.e., those solutions that evolve by translations of $\mathbb{R}^{n+1}$ in a given direction $v$. In this case, the immersion $f: M^n \rightarrow \mathbb{R}^{n+1}$ satisfying \eqref{eq:soliton-general} with $\xi=v$ is called \emph{translator}.  Special attention has been given to translators in $\HitR$ where the translation is given in the direction of the factor $\mathbb{R}$. Namely, those surfaces $M$ in $\mathbb{H}^2 \times \mathbb{R}$ satisfying 
\begin{equation}
H = \langle N, \partial_h \rangle. \label{eq:translators}
\end{equation}
where $\partial_h$ denotes the standard coordinate vector field of the factor $\mathbb{R}$. Trivial examples of translators are $\gamma\times\mathbb R$, where $\gamma$ is any geodesic of $\mathbb H^2$. Non-trivial examples of translators satisfying \eqref{eq:translators} are the symmetric ones, i.e., those translators invariant under isometries of $\mathbb{H}^2 \times \mathbb{R}$ that leaves $\mathbb{R}$ fixed. Such examples are given in details in Section \ref{symmetric-translators}.

In this work, we focus on complete properly immersed translators in the product space $\HitR$ satisfying equation \eqref{eq:translators}, analyzing their asymptotic behavior at infinity.

\subsection{Tangency principle}

It is well-known that translators in the Euclidean space $\mathbb{R}^3$ can be regarded as minimal in a conformally flat metric. As a consequence, a tangency principle holds for translators (see Theorem 2.3 in \cite{martin-topology}). Similarly, as stated in Theorem 2.1 of \cite{bueno-2018} (vertical) translators in $\HitR$ are minimal surfaces with the conformally changed metric $e^h \langle , \rangle$, where $h$ is the height function of $\HitR$. This allows an analogous tangency principle for such surfaces (Theorem 2.2 of \cite{bueno-2018}):

\begin{theorem}\label{tangency}
Let $\Sigma_1$ and $\Sigma_2$ be two connected translators in $\HitR$ with possibly non-empty boundaries $\partial \Sigma_1$, $\partial \Sigma_2$. Suppose that one of the following statements holds
\begin{itemize}
\item There exists $p \in int(\Sigma_1) \cap int(\Sigma_2)$ with $(N_1)_p = (N_2)_p$, where $(N_i)_p$ is the unit normal of $\Sigma_i$ at $p$.
\item There exists $p \in \partial \Sigma_1 \cap \Sigma_2$ with $(N_1)_p = (N_2)_p$ and $(\xi_1)_p = (\xi_2)_p$, where $(\xi_i)_p$ is the interior unit conormal of $\partial \Sigma_i$ at $p$.
\end{itemize}
 Assume that $\Sigma_1$ lies locally around $p$ at one side of $\Sigma_2$. Then, in either situation, both surfaces agree in a neighborhood of $p$. Moreover, if both surfaces $M_i$ are complete, then $M_1 = M_2$.
\end{theorem}

\section{Symmetric translators} \label{symmetric-translators}

In what follows we present three classes of symmetric translators of $\HitR$, following the terminology of \cite{ lima-pipoli,umbilical}. In the hyperbolic space $\Hi$, we consider three special types of one-parameter families of isometries: the rotations around a fixed point (elliptic isometries), the translations along horocycles sharing the same point at infinity (parabolic isometries), and the translations along a fixed
geodesic (hyperbolic isometries). A nice property of these isometries is the fact that each one fixes a family of curves in $\Hi$ with constant geodesic curvature. Indeed, an elliptic isometry fixes a family of concentric geodesic circles, whereas a parabolic (resp. hyperbolic) translation fixes a family
of parallel horocycles (resp. lines equidistant to a given geodesic or equidistant lines, for short). We extend these isometries of $\Hi$ in a natural way to isometries of $\HitR$ by fixing the factor $\mathbb{R}$ pointwise, and we will keep the terminology for each of these extensions, i.e., elliptic, parabolic and hyperbolic isometries of $\HitR$. In this context, a surface $\Sigma$ is called symmetric if it is invariant under the action of such groups of isometries. We remark the property that each symmetric surface is foliated by vertical translations of its corresponding family curves of constant geodesic curvature. Finally, we call $\Sigma$ rotational, if it is invariant under elliptic isometries, and parabolic (resp. hyperbolic) if it is invariant by horizontal parabolic translations (resp. hyperbolic translations). Symmetric translators of $\HitR$ where considered in \cite{bueno-2018, bueno-lopez-2025, lima-pipoli, lira-martin-2019}, and they are given as follows.

\begin{itemize}
    \item Translators of $\HitR$ invariant under parabolic translations (Theorem 2 in \cite{lima-pipoli} and Theorem 11 in \cite{lira-martin-2019}): the \emph{parabolic translating catenoid} belongs to a family of properly embedded translators $\left\{ \Sigma_{\lambda},\,\lambda>0 \right\}$ which are homeomorphic to $\mathbb{R}^2$ and are given by the union of two graphs $\Sigma^+$ and $\Sigma^{-}$ over the complement of a horoball bounded by a horocircle $\mathcal{H}_{\lambda}$ which have unbounded height and satisfy $\partial \Sigma^{\pm} = \mathcal{H}_0$. Its set of points of minimal height is a horocircle contained in a horizontal hyperplane $\Hi \times \{t_0\}$. The \emph{parabolic bowl} is an entire vertical graph with unbounded height from above and from below and it has constant mean curvature. See Figure \ref{fig:parabolic}. Using the nomenclature in Section \ref{geocomp}, the vertical asymptotic boundary of the parabolic bowl is $\{p\}\times\mathbb R$, while its corner asymptotic boundary is 
    $[\partial_{\infty}\mathbb H^2\times\{+\infty\}]\cup[\{p\}\times\{-\infty\}]$.  The vertical asymptotic boundary of the parabolic translating catenoids is $\{p\}\times[T,+\infty)$, while its corner asymptotic boundary is 
    $\partial_{\infty}\mathbb H^2\times\{+\infty\}$. Finally, the horizontal asymptotic boundary is empty for each of these examples.

    \begin{figure}[htbp!]
\begin{minipage}{0.45\linewidth}
    \begin{center}
\includegraphics[width=0.8\linewidth]{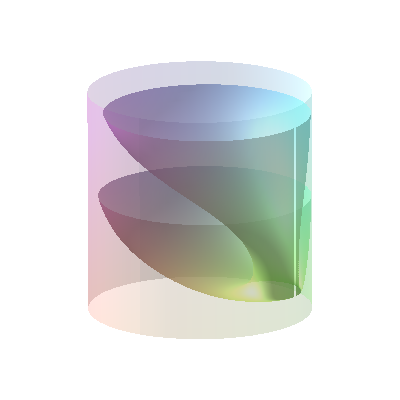}
\end{center}
\end{minipage}
\begin{minipage}{0.45\linewidth}
\begin{center} \vspace{0.1cm}
\includegraphics[width=0.8\linewidth]{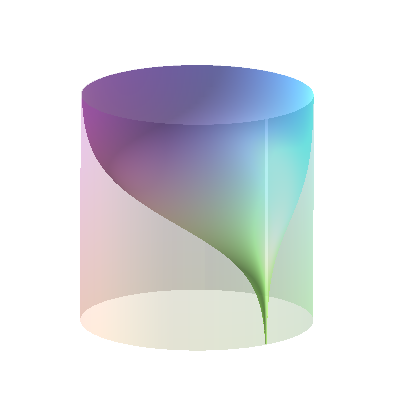}
\end{center}
\end{minipage}
\caption{Parabolic translating catenoid (left); Parabolic bowl (right)}
\label{fig:parabolic}
\end{figure}

\end{itemize}

\begin{itemize}
    \item Translators of $\HitR$ invariant under elliptic translations (Theorem 1 in \cite{lima-pipoli}, Theorem 7 in \cite{lira-martin-2019}, and Theorems 3.3 and 3.4 in \cite{bueno-2018}): the \emph{rotational translating catenoid} belongs to a family of a properly embedded annular rotational translators $\left\{ \Sigma_{\lambda},\,\lambda>0 \right\}$ given by the union of two graphs $\Sigma^+$ and $\Sigma^{-}$ over the complement of a geodesic circle $B_{\lambda}(o) \subset \Hi$ which have unbounded height and satisfy $\partial \Sigma^{\pm} = \partial B_{\lambda}(o)$. Its set of points of minimal height is a geodesic circle centered at the axis of rotation which is contained in a horizontal hyperplane $\Hi \times \{t_0\}$. The \emph{rotational bowl} is a strictly convex entire vertical graph with unbounded height and contained in a closed half-space $\Hi \times [t_0,+\infty).$ See Figure \ref{fig:rotational}. The asymptotic boundary of those translators contains just corner points, in particular it is $\partial_{\infty}\mathbb H^2\times\{+\infty\}$.

\end{itemize}

    \begin{figure}[htbp!]
\begin{minipage}{0.45\linewidth}
    \begin{center}
\includegraphics[width=0.8\linewidth]{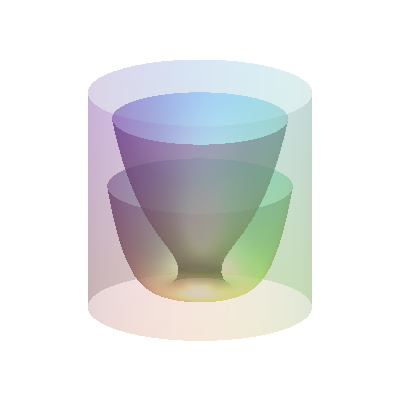}
\end{center}
\end{minipage}
\begin{minipage}{0.45\linewidth}
\begin{center} \vspace{0.1cm}
\includegraphics[width=0.8\linewidth]{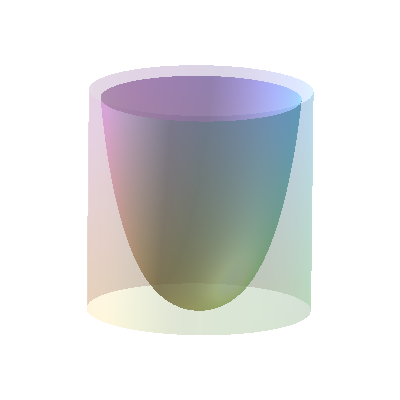}
\end{center}
\end{minipage}
\caption{Rotational translating catenoid (left); Rotational bowl (right)}
\label{fig:rotational}
\end{figure}

\begin{itemize}
   
\item  Translators of $\HitR$ invariant under hyperbolic translations (Theorem 3 in \cite{lima-pipoli}, Theorem 12 in \cite{lira-martin-2019}, and Theorem 4.4 in \cite{bueno-lopez-2025}). Two types of translators exhibit such symmetries. The first one is called \emph{hyperbolic bowl}: it is an entire vertical graph with unbounded height and contained in a closed half-space $\Hi \times [t_0,+\infty),$ tangent to an equidistant line contained in the horizontal hyperplane $\Hi \times \{t_0\}$. The second type consists of a one-parameter family of properly embedded translators with the following properties. Denoting by $\mathcal{E}_{r}$ the curve equidistant from a fixed geodesic $\mathcal{E}_0$, at a distance $\left|r\right|$, for any $r_0$ there exists a translator invariant under hyperbolic translation given by the union of two graphs, both unbounded from above and defined over the domain of $\mathbb H^2$ foliated by $\{\mathcal{E}_{r} | r>r_0\}$. Moreover, any of such surfaces is tangent to $\mathcal E_{r_0}\times\mathbb{R}$, it is contained in $\mathbb H^2\times[t_0,\infty)$ and it is tangent to $\mathbb{H}^2\times\{t_0\}$ along a $\mathcal E_{\rho}$,  where $\rho=\rho(r_0)$ and $t_0$ can be chosen freely up to vertical translations. In particular, any element of this family is homeomorphic to $\mathbb{R}^2$. See Figure \ref{fig:equidistant} and \ref{fig:A}.

%\begin{comment}
\begin{figure}[htbp!]
\begin{minipage}{0.31\linewidth}
    \begin{center}
\includegraphics[width=1\linewidth]{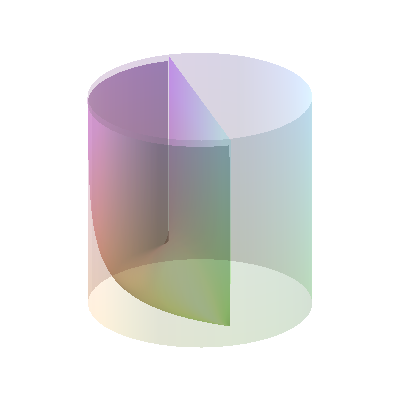}
\end{center}
\end{minipage}
\begin{minipage}{0.31\linewidth}
    \begin{center}
\includegraphics[width=1\linewidth]{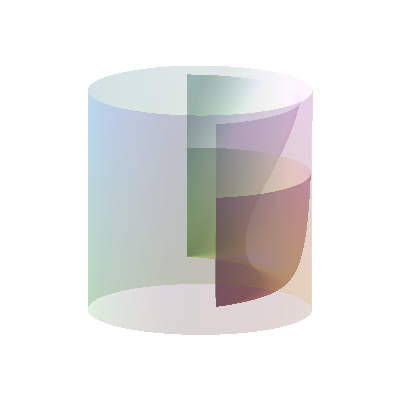}
\end{center}
\end{minipage}
\begin{minipage}{0.31\linewidth}
\begin{center} \vspace{0.1cm}
\includegraphics[width=1\linewidth]{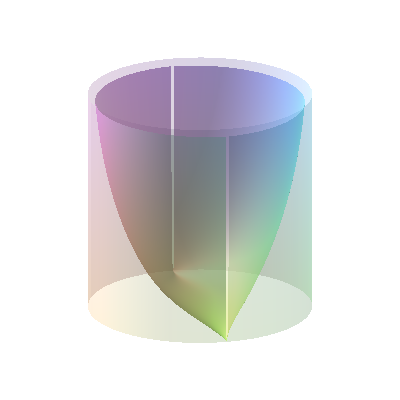}
\end{center}
\end{minipage}
\caption{Using the nomenclature of \cite{lima-pipoli}: hyperbolic translating catenoid with $r_0=0$ (left), $r_0>0$ (center) and hyperbolic bowl (right)}
\label{fig:equidistant}
\end{figure}
%\end{comment}

Note that the construction is independent on the choices of the initial geodesic $\mathcal E_0$ and the minimal height $t_0$, in fact, the image of a translator under any isometry of the ambient space is still a translator.
This family of surfaces will play a crucial role in the proofs that follow, so we reserve the symbol $\Omega$ for it and $\Sigmaw$ for its elements. These translators are known in the literature by at least two different names: they are called \emph{hyperbolic translating catenoids} in \cite{lima-pipoli} and \emph{hyperbolic v-grim reapers} in \cite{bueno-lopez-2025}. 

%\begin{comment}
\begin{figure}[htbp!]
\begin{minipage}{0.45\linewidth}
    \begin{center}
\includegraphics[width=0.8\linewidth]{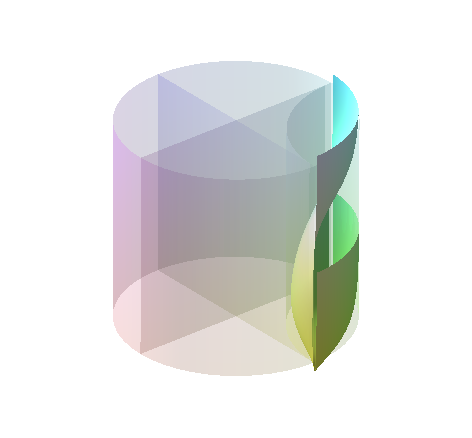}
\end{center}
\end{minipage}
\begin{minipage}{0.45\linewidth}
\begin{center} \vspace{0.1cm}
\includegraphics[width=0.6\linewidth]{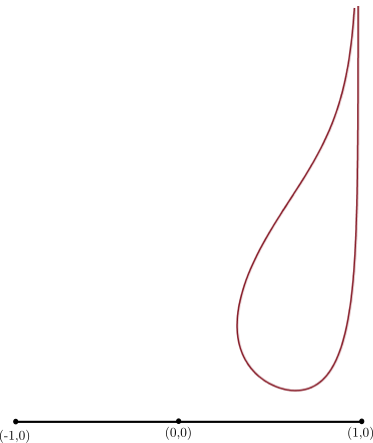}
\end{center}
\end{minipage}
\caption{An element of $\Omega$, $r_0>0$ (left) and $\Sigmaw\cap [c\times\mathbb R]$, $\Sigmaw\in\Omega$, where $c$ is the geodesic in $\mathbb{H}^2$ whose end points are $(1,0)$ and $(-1,0)$ and $\Sigmaw$ is determined by the geodesic whose end points are $(0,1)$ and $(0,-1)$, tangent to the horizontal plane $\mathbb{H}^2 \times \left\{0 \right\}$ and to the vertical plane $\mathcal{E}_r$ (right)} \label{fig:A}
\end{figure}
%\end{comment}

For any translator invariant under hyperbolic translation, the vertical asymptotic boundary consists in two half lines $\{p_i\}\times[t_0,+\infty)$, where the $p_i$ are the end points of the geodesic $\mathcal E_0$. The corner asymptotic boundary is the whole $\partial_{\infty}\mathbb H^2\times\{+\infty\}$ for the hyperbolic bowl and a connected arc with end points $\{p_i\}\times\{+\infty\}$ in the other cases. Finally, the horizontal asymptotic boundary is empty except in the special case $r_0=0$ where it is $\mathcal E_0\times\{+\infty\}$.
\end{itemize}

\section{The vertical asymptotic boundary.}\label{finite}
In this section, we study the behavior of the vertical asymptotic boundary of a translator. The first bullet in Theorem~\ref{main} is a consequence of Theorem \ref{finiteboundary}. First, we need this definition.

\begin{definition}\label{cylinder}
    Given $\varepsilon>0$ and $(p,0) \in \partial_{\infty} \Hi \times \left\{ 0 \right\}$, let $q_1, q_2 \in \partial_{\infty} \Hi \times \left\{ 0 \right\}$ be the two distinct points such that the Euclidean distance from $p$ to $q_i$ is $\varepsilon$, $i=1,\,2$, and let $c$ be the complete geodesic with asymptotic boundary $\left\{q_1, q_2\right\}$. Then we define $\mathcal{S}(p,\varepsilon)$ to be the vertical geodesic plane $c \times \mathbb{R}$, see Figure~\ref{fig:B}. Furthermore, we define $\widetilde{\mathcal{S}}(p,\varepsilon)$ to be the connected component of $[\Hi \times \mathbb{R}]\setminus \mathcal{S}(p,\varepsilon)$ that contains $p$ in its boundary at infinity.
\end{definition}

%\begin{comment}
\begin{figure}[htbp!]
\begin{center}
\includegraphics[width=1\linewidth]{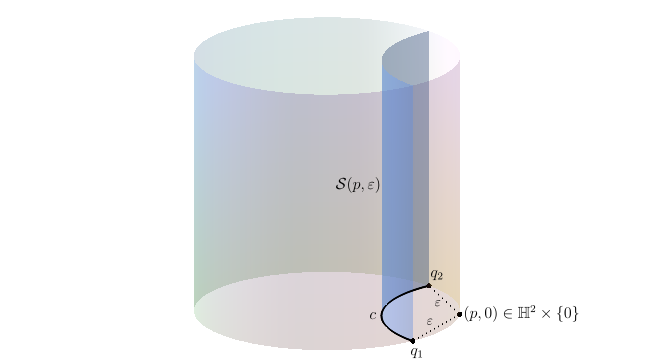}
\end{center}
\caption{$\mathcal{S}(p,\varepsilon)$, as in Definition \ref{cylinder}}\label{fig:B}
\end{figure}
%\end{comment}

Given $(p,t) \in \partial_{\infty} \Hi \times \mathbb{R}$, we denote by $l_p$ the vertical line $\{p\}\times \mathbb{R}\subset \partial_{\infty} \Hi \times \mathbb{R} $.

\begin{theorem}\label{finiteboundary}
Let $M$ be  a properly immersed translator in $\HitR$ with (possibly empty) compact boundary. Suppose that its vertical asymptotic boundary, $\partial_\infty^v M$, is a proper family of disjoint immersed continuous curves and let $\gamma$ be a connected component of $\partial_\infty^v M$ so that $M\cup \gamma $ is a continuous surface with boundary. Then $\gamma=l_p$ or $\gamma=\{p\}\times [T,\infty)$ for some $T\in \mathbb{R}$.
\end{theorem}
\begin{proof}
If the compact boundary of $M$ is non-empty, fix $\nu>0$ such that the infimum of the Euclidean distance between $\partial M$ and the boundary cylinder is greater than $\nu$. We first prove that $\gamma$ must be contained in a vertical line.

\begin{lemma}\label{vertical}
There exists $p\in \partial_\infty \mathbb{H}^2$ such that $\gamma\subset l_p$.
\end{lemma}
\begin{proof}
Arguing by contradiction, let's 
assume that $\gamma$ is not contained in a vertical line and let $(p,t)\in \gamma$ be a point that is not contained in a vertical segment, that is there exists $\rho>0$ such that $\{(p,t)\}=\gamma \cap [\{p\}\times (t-\rho, t+\rho)]$. Let $\varepsilon\in(0,\frac \nu 2)$ be a small real number to be chosen later. Note that $\varepsilon<\frac \nu 2$ gives that $\partial M$, the finite boundary,  is NOT contained in $\widetilde{\mathcal{S}}(p,\varepsilon)$. Let $M_\varepsilon$ be the  connected component of $M\cap \widetilde{\mathcal{S}}(p,\varepsilon)$ that contains $(p,t)$ in its asymptotic boundary and let $\gamma_\varepsilon\subset \gamma$ be the connected component of the asymptotic boundary of $M_\varepsilon$ containing $(p,t)$.  Note that since  $\partial_\infty^v M$ is a proper family of disjoint curves, there exist  $\varepsilon_1>0$ and $\delta\in(0,\rho)$ such that for any $\varepsilon\leq \varepsilon_1$, $\gamma_\varepsilon$  is the only component of $\partial_\infty^v M$ that intersects the closed slab $\partial_\infty \mathbb{H}^2\times [t-\delta,t+\delta]$. 

\begin{claim}\label{openslab}
There exists $0<\overline{\varepsilon}< \varepsilon_1$ such that
\[M_{\overline{\varepsilon}}\subset \mathbb{H}^2\times (t-\delta,t+\delta).\]  In particular, $\gamma_{\overline{\varepsilon}}$ is the asymptotic boundary of $M_{\overline{\varepsilon}}$ and
\[\gamma_{\overline{\varepsilon}}\subset \widetilde{\mathcal{S}}(p,\overline{\varepsilon})\cap [\mathbb{H}^2\times (t-\delta,t+\delta)].\]
See Figure~\ref{fig:C1}.
\end{claim}

%\begin{comment}
\begin{figure}[htbp!]
\begin{center}
\includegraphics[width=0.6\linewidth]{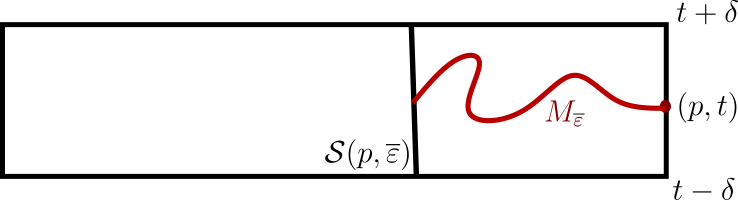}
\end{center}
\caption{A vertical section of $\mathbb H^2\times(t-\delta,t+\delta)$ describing Claim \ref{openslab}.}
\label{fig:C1}
\end{figure}
%\end{comment}

\begin{proof}
Arguing by contradiction, assume that for any $\varepsilon< \varepsilon_1$,
$M_\varepsilon$ is not contained in $\mathbb{H}^2\times (t-\delta,t+\delta)$. Without loss of generality, let's assume that for any $\varepsilon< \varepsilon_1$ we have that \[M_\varepsilon\cap [\mathbb{H}^2\times \{t+\delta\}]\neq \emptyset.\] 
Since $M_\varepsilon\cap [\mathbb{H}^2\times \{t+\delta\}]\subset \widetilde{\mathcal{S}}(p,\varepsilon)$, by definition, this implies that $(p, t+\delta)$ is a point in the asymptotic boundary of $M$. In particular, since for any $\varepsilon< \varepsilon_1$, $M_\varepsilon\subset M_{\varepsilon_1}$, the point $(p, t)$ is a point in the asymptotic boundary of $M_{\varepsilon_1}$. By our choice of $\varepsilon_1>0$ and $\delta\in(0,\rho)$, $\gamma_{\varepsilon_1}$ is the only component of $\partial_\infty^v M$ that intersects the closed slab $\partial_\infty \mathbb{H}^2\times [t-\delta,t+\delta]$. Then $(p, t+\delta)\in \gamma_{\varepsilon_1}$. This leads to a contradiction because $\delta<\rho$  and $(p,t)$ is the only point in the intersection $\gamma_{\varepsilon_1} \cap [\{p\}\times (t-\rho, t+\rho)]$. This contradiction finishes the proof of the claim.
\end{proof}

With $\delta$ and $\overline{\varepsilon}$ given by the previous discussion and Claim~\ref{openslab}, there are two cases to consider and rule out (see Figure \ref{fig:C2}):
\begin{enumerate}
    \item $\gamma_{\overline{\varepsilon}}$ is  on one side of $l_p$.
    \item $\gamma_{\overline{\varepsilon}}$ is NOT  on one side of $l_p$.
\end{enumerate}

%\begin{comment}
\begin{figure}[htbp!]
\begin{center}
\includegraphics[width=0.6\linewidth]{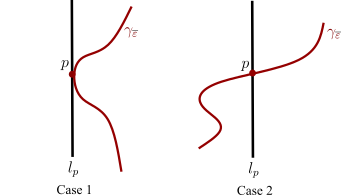}
\end{center}
\caption{On the left, Case 1, on the right Case 2.}
\label{fig:C2}
\end{figure}
%\end{comment}

Without loss of generality, after applying a vertical translation, let's assume that $(p,t)=(p,0)$.
We first rule out Case 1.

\begin{claim}\label{no-case1}
    Case 1. does not hold.
\end{claim}

\begin{proof}
  In this case, the vertical projection of $\gamma_{\overline{\varepsilon}}$ on $\partial_{\infty} \Hi \times \left\{ 0 \right\}$ is
an arc $\beta$ with $p$ as one of the two end points.   See Figure~\ref{fig:D}. Let $q_\beta$ denote the other endpoint of $\beta$, let $q_0\in \partial_{\infty} \Hi \times \left\{ 0 \right\}$ be the point so that $p$ is the midpoint of the arc in $\partial_{\infty} \Hi \times \left\{ 0 \right\}$ whose endpoints are $q_\beta$ and $q_0$, and let $\beta_0$ be the (small) open arc whose endpoints are $p$ and $q_0$.

%\begin{comment}
\begin{figure}[htbp!]
\begin{center}
\begin{minipage}[b]{0.5\textwidth}
    \centering
    \includegraphics[width=\textwidth]{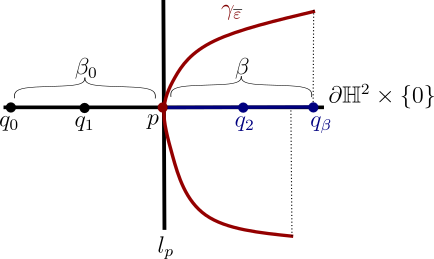}
  \end{minipage}
  \hfill
  \begin{minipage}[b]{0.45\textwidth}
    \centering
    \includegraphics[width=\textwidth]{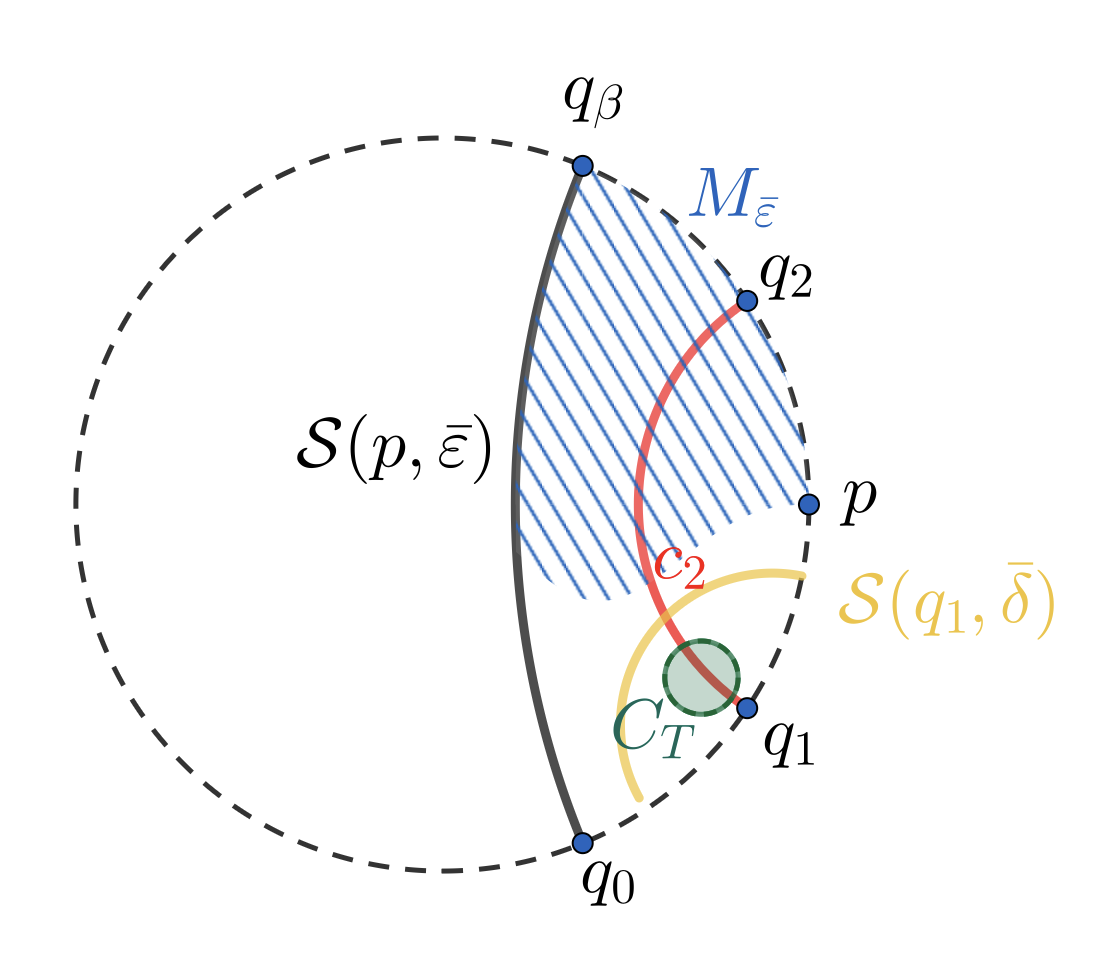}
  \end{minipage}
\end{center}
\caption{The constructions in the proof of Claim \ref{no-case1}. On the left, front view; on the right, top view.}
\label{fig:D}
\end{figure}
%\end{comment}

Using these notations, $\gamma_{\overline{\varepsilon}}$ is on the side of $l_p$ that does not ``contain'' $\beta_0$. For any $q\in\beta_0$ then $l_q\cap \gamma_{\overline{\varepsilon}}=\emptyset$. 
Let $q_1$ be the midpoint of $\beta_0$.  Then, since $l_{q_1}\cap \gamma_{\overline{\varepsilon}}=\emptyset$, arguing as in the proof of Claim~\ref{openslab}, there exists $\overline{\delta}<\frac{\overline{\varepsilon}}{10}$ such that   \[
\widetilde{\mathcal{S}}(q_1,\overline{\delta})\cap M_{\overline{\varepsilon}}=\emptyset.\]

Let $C \subset \HitR$ be the rotational bowl whose lowest point is $(0,-2\delta)$, see Section~\ref{symmetric-translators}. Then, we can apply a hyperbolic translation $T$ such that $C_T:=T(C)\cap [\Hi \times [-2\delta, 2\delta]]$ is contained in $
\widetilde{\mathcal{S}}(q_1,\overline{\delta})$. In particular, $C_T\cap M_{\overline{\varepsilon}}=\emptyset$ and its finite boundary is contained in $\Hi \times \{2\delta\}$, see Figure \ref{fig:pipoli}.

Let $c_2 \subset \Hi \times \left\{ 0 \right\}$ be the complete geodesic with asymptotic boundary points $q_1$ and $q_2$, where $q_2$ is the midpoint of $\beta$. Consider hyperbolic translations along $c_2$. Recall that by Claim~\ref{openslab}, 
\[\partial M_{\overline{\varepsilon}}\subset \mathcal{S}(p,\overline{\varepsilon})\cap [\mathbb{H}^2\times [-\delta,\delta]]\]
and observe that all translated copies of $C_T$ are contained in $\widetilde{\mathcal{S}}(p,\frac 34 \overline{\varepsilon})$ and have their finite boundary contained in $\Hi \times \{2\delta\}$. This implies that the boundary of any translated copy of $C_T$ has no intersection with $M_{\overline{\varepsilon}}$ and the boundary of $M_{\overline{\varepsilon}} $ has no intersection with any translated copy of $C_T$. Therefore, since $C_T\cap M_{\overline{\varepsilon}}=\emptyset$, some translated copy of $C_T$, $\widetilde{C_T}$, must obtain its first point of contact with $M_{\overline{\varepsilon}}$ at an interior point of $M_{\overline{\varepsilon}}$ and $\widetilde{C_T}$. This contradicts the tangency principle, Theorem~\ref{tangency}, and concludes the proof of the claim.
\end{proof}

%\begin{comment}
\begin{figure}[htbp!]
\begin{center}
\includegraphics[width=0.5\linewidth]{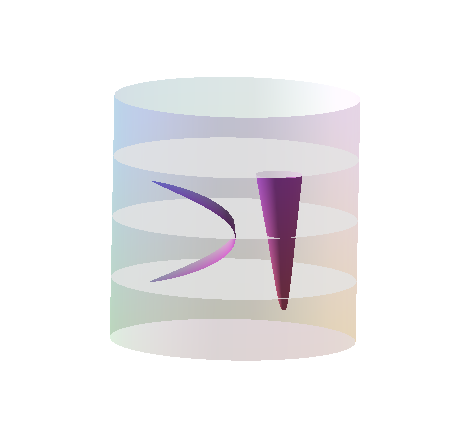}
\caption{Rotational bowl as barrier}
\end{center}
\label{fig:pipoli}
\end{figure}
%\end{comment}

We next rule out Case 2.

\begin{claim}\label{no-case2}
    Case 2. does not hold.
\end{claim}
\begin{proof}
In this case, the vertical projection of $\gamma_{\overline{\varepsilon}}$ on $\partial_{\infty} \Hi \times \left\{ 0 \right\}$ is
an arc $\beta$ with $p$ as one of its interior points. See Figure~\ref{fig:E}.

%\begin{comment}
\begin{figure}[htbp!]
\begin{center}
\includegraphics[width=0.6\linewidth]{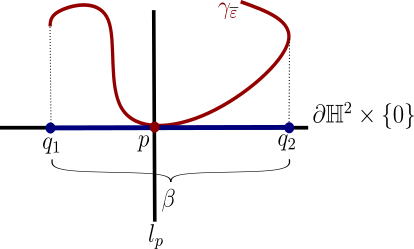}
\end{center}
\caption{The construction in the proof of Claim \ref{no-case2}.}
\label{fig:E}
\end{figure}
%\end{comment}

Let $q_1$ and $q_2$ denote the endpoints of $\beta$, let $\alpha>0$ be the distance between $p$ and $q_1$ and assume that $q_1$ is closer to $p$ than $q_2$. Let $M_{\alpha}\subset M_{\overline{\varepsilon}}$ denote the connected component of $M_{\overline{\varepsilon}}\cap \widetilde{\mathcal{S}}(p,\alpha)$ that contains $(p,0)$ in its asymptotic boundary. Recall that $M_\alpha\subset \Hi \times (-\delta,\delta)$ and that $\partial M_\alpha\subset \mathcal{S}(p,\alpha)$.

By definition, there exists a sequence of points $(p_n,t_n)_{n\in\mathbb N}\subset M_{\alpha}$, converging to $(p,0)$. Let $\omega_n$ be the complete geodesic that contains the points $(p_n, t_n)$ and $(\vec0,t_n)$ in $\Hi\times\{t_n\}$ (where $\vec0$ denotes the origin of $\Hi$) and consider the hyperbolic translation, $T_n$ along $\omega_n$ that takes $(p_n, t_n)$ to $(\vec0,t_n)$. See Figure~\ref{fig:F}.

%\begin{comment}
\begin{figure}[htbp!]
\begin{minipage}{0.45\linewidth}
    \begin{center}
\includegraphics[width=0.8\linewidth]{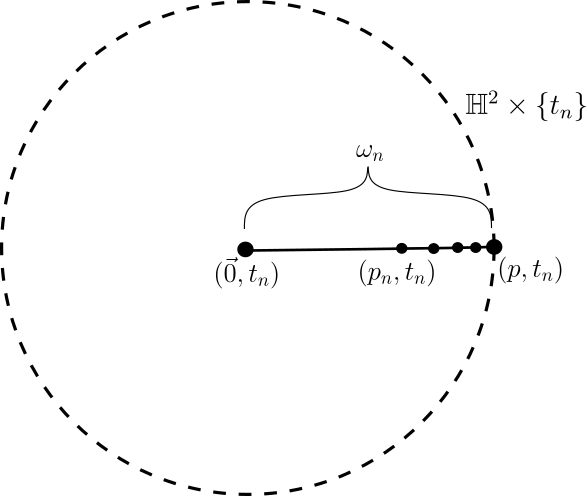}
\end{center}
\end{minipage}
\begin{minipage}{0.45\linewidth}
\begin{center} \vspace{0.1cm}
\includegraphics[width=1.5\linewidth]{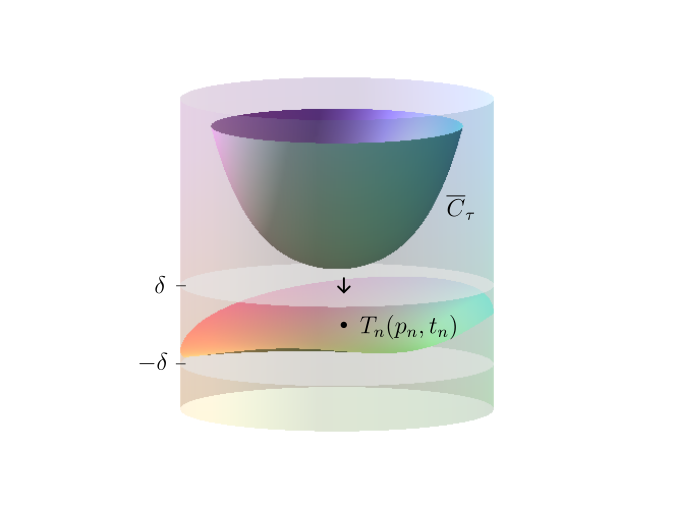}
\end{center}
\end{minipage}
\caption{The stretching argument and the rotational bowl}\label{fig:F}
\end{figure}
%\end{comment}

Note that $(\vec0,t_n)\in T_n(M_\alpha)\subset \Hi \times (-\delta,\delta)$ and that 
given $\mu>0$ there exists $n_\mu\in\mathbb{N}$ such that for any $n\geq n_\mu$, the vertical projection of $\partial T_n(M_\alpha)$ in $\Hi \times \left\{ 0 \right\}$ does not intersect the disk centered at the origin of radius $1-\mu$.

Let $C \subset \HitR$ be the rotational bowl, see Section~\ref{symmetric-translators}, whose lowest point is $(0,-2\delta)$, let 
\[\overline C := C\cap [\Hi \times [-2\delta,2\delta]]\]
and let $\overline C_\tau$ be $\overline C$  vertically translated upward by $\tau$ so that 
$\overline C_0=\overline C$. For any $\tau\in [0,4\delta]$, the vertical projection of $\partial \overline C_\tau$ in $\Hi \times \left\{ 0 \right\}$ is a circle of radius $1-r$, $r\in\mathbb{R}$. Fix $\mu<\frac r2$ and note that with this choice of $\mu$, the boundary of $T_{n_\mu}(M_\alpha)$ is disjoint from any vertical translated $\overline C_\tau$. Note also that for any $\tau\in [0,4\delta]$, the boundary of $\overline C_\tau$ is disjoint from $T_{n_\mu}(M_\alpha)$.

Therefore, since $\overline C_0 \cap T_{n_\mu}(M_\alpha)\neq\emptyset$ but $\overline C_{4\delta} \cap T_{n_\mu}(M_\alpha)=\emptyset$, then some translated copy  $\overline C_{\overline \tau} $,  $\overline{\tau}\in (0,4\delta)$ must obtain its last point of contact with $T_{n_\mu}(M_\alpha)$ at an interior point of $T_{n_\mu}(M_\alpha)$ and $\overline C_{\overline \tau}$. This contradicts the tangency principle, Theorem~\ref{tangency}, and concludes the proof of the claim.
\end{proof}

The fact that neither Case 1. nor Case 2. can hold is clearly a contradiction. This contradiction completes the proof of Lemma~\ref{vertical}. 
\end{proof}

In light of Lemma~\ref{vertical} we find that there exists $p\in \partial_\infty \mathbb{H}^2$ such that $\gamma\subset l_p$ and we need to prove that $\gamma=l_p$ or $\{p\}\times [T,\infty)$ with $T\in \mathbb{R}$. Arguing by contradiction, assume that 
\[
\tau:=\sup_{(p,t)\in\gamma}{t}<\infty.
\]
After applying a vertical translation, we can assume that $\tau=0$.  
Given $\varepsilon>0$ let $\widetilde{M}_\varepsilon$ be a connected component of $M\cap \widetilde{\mathcal{S}}(p,\varepsilon)\cap [\mathbb{H}^2\times [-10,+\infty)]$ that contains $(p,0)$ in its asymptotic boundary.

\begin{claim}\label{openslab2}
There exists $0<\overline{\varepsilon}<\frac \nu2$ such that 
\[\widetilde{M}_{\overline{\varepsilon}}\subset \mathbb{H}^2\times [-10,1)\] and $\gamma\cap[\mathbb{H}^2\times [-10,1)]$ is the only component of the asymptotic boundary of $\widetilde{M}_{\overline{\varepsilon}}$. In particular,
\[\partial \widetilde{M}_{ \overline{\varepsilon}}\subset [\mathcal{S}(p,\overline{\varepsilon})\cap \mathbb{H}^2\times [-10,1)]\cup [\mathbb{H}^2\times \{-10\}].\]
\end{claim}
\begin{proof}
Note that \[
\gamma\cap [\{p\}\times (0, 1)]=\emptyset.
\]
Using this observation, the proof follows more or less the same arguments as in the proof of Claim~\ref{openslab} and uses the fact that $\partial^v_\infty M$ is a proper family of disjoint immersed continuous curves. Note that in this case we are only saying that it is contained in a ``half-open'' slab.
\end{proof}
Fix $\overline{\varepsilon}>0$ as given by Claim~\ref{openslab2}. For any $\varepsilon>0$, let $q^\varepsilon_1$ and $q^\varepsilon_2$ be the points in $\partial_\infty \mathbb{H}^2\times\{0\}$ at distance $\varepsilon$ from $p$.  Fix any $\Sigmaw\in\Omega$, see Section~\ref{symmetric-translators} for the definitions of $\Sigmaw$ and $\Omega$, and let $\Sigmaw_\varepsilon$ be an isometric copy of $\Sigmaw$ with asymptotic boundary $\{q^\varepsilon_1\}\times [-2,\infty)\cup \{q^\varepsilon_2\}\times [-2,\infty)$. Note that as $\varepsilon$ goes to zero, $\Sigmaw_\varepsilon$ converges continuously to $\{p\}\times [-2,\infty)$. In particular, if $\varepsilon\leq \frac{\overline{\varepsilon}}2$, then $\Sigmaw_\varepsilon \cap \partial \widetilde{M}_{\overline{\varepsilon}}=\emptyset$.

Recall that $\Sigmaw_\varepsilon$ separates $\mathbb{H}^2 \times \mathbb{R}$ into two components. Let $V^\varepsilon_{in}$ denote the component that is bounded from below, and let $V^\varepsilon_{out}$ denote the other component. 

 By definition, there exists a sequence of points $(p_n,t_n)_{n\in\mathbb N}\subset\widetilde{M}_{\overline{\varepsilon}}$ that converge to the point $(p,0)$. After passing to a subsequence, we can assume that $t_n>-1$. Moreover, by taking $\overline{n}$ sufficiently large we can assume that the following holds: the point $(p_{\overline{n}}, t_{\overline{n}})$ is contained in $ V^{\frac{\overline{\varepsilon}}2}_{out}$ and if we let $c_{\overline{n}}$ be the geodesic arc whose endpoints are $(0,t_{\overline{n}})$ and $(p_{\overline{n}}, t_{\overline{n}})$, the intersection of $c_{\overline{n}}$ with $\Sigmaw_{\frac{\overline{\varepsilon}}2}$ consist of two points. See Figure~\ref{fig:I}.

%\begin{comment}
\begin{figure}[htbp!]
\begin{center}
\includegraphics[width=1\linewidth]{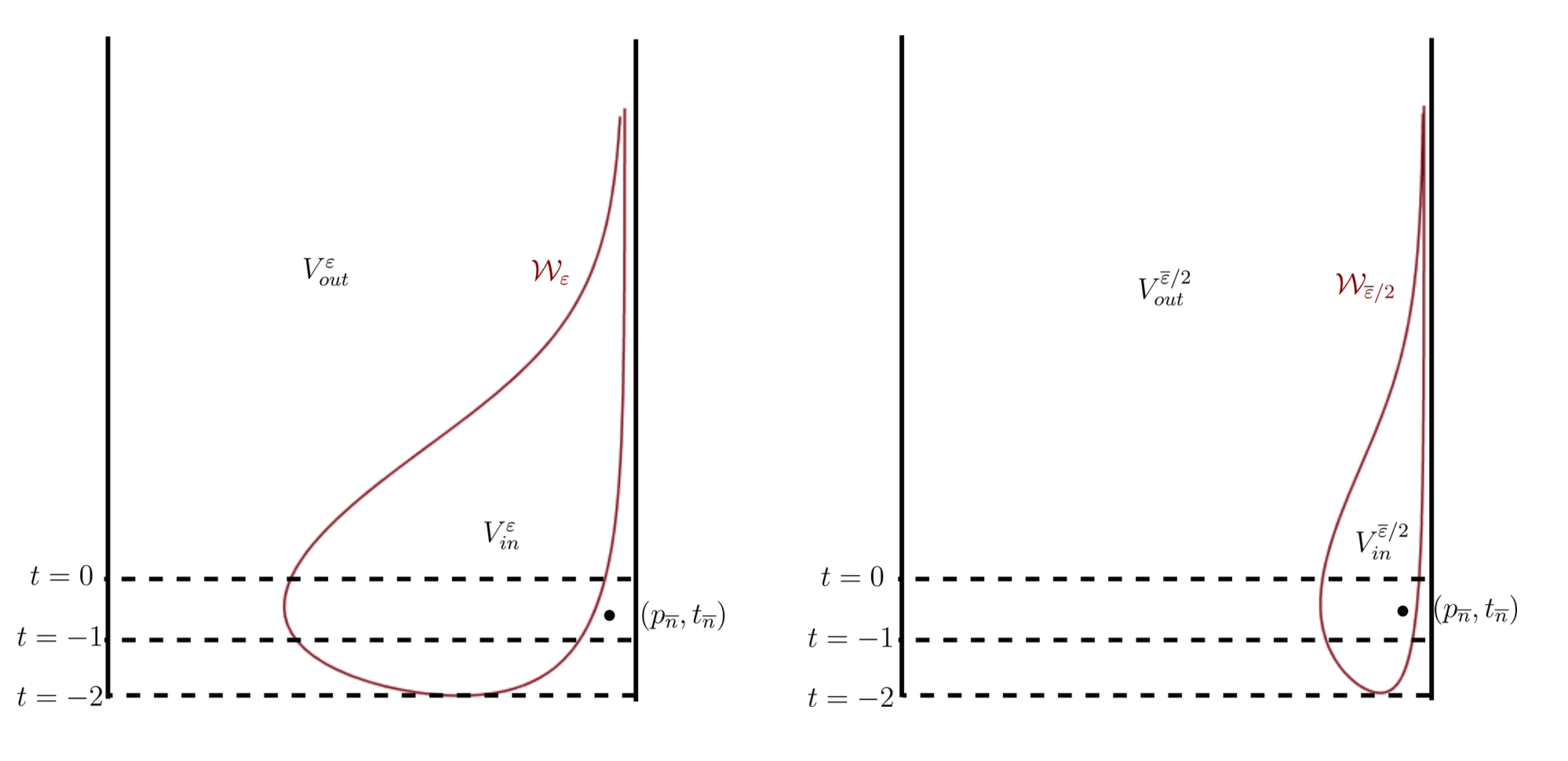}
\end{center}
\caption{The points $(p_{\overline{n}}, t_{\overline{n}})$ are ``inside'' $\Sigmaw_{\frac{\overline{\varepsilon}}2}$ }
\label{fig:I}
\end{figure}
%\end{comment}

By continuity, there exists $\varepsilon_1<\frac{\overline{\varepsilon}}2$ such that $(p_{\overline{n}}, t_{\overline{n}})\in V^{\varepsilon_1}_{in}$. However since $(p_n,t_n)$ converges to $(p,0)\notin V^{\varepsilon_1}_{in}$, then there are points $(p_n,t_n)\notin V^{\varepsilon_1}_{in}$ for $n$ large, giving that
 \[
 \Sigmaw_{\varepsilon_1}\cap \widetilde{M}_{\overline{\varepsilon}}\neq \emptyset.
 \]

Let $\Sigmaw_{\varepsilon_1}^\tau$ be $\Sigmaw_{\varepsilon_1}$ vertically translated upward by $\tau\in [0, 5]$ so that
$\Sigmaw_{\varepsilon_1}^0=\Sigmaw_{\varepsilon_1}$. By construction, for any $\tau\in [0, 5]$, 
\[
\Sigmaw_{\varepsilon_1}^\tau\cap \partial \widetilde{M}_{\overline{\varepsilon}}=\emptyset.
\]
This observation, together with the fact that $\Sigmaw_{\varepsilon_1}^0\cap \partial \widetilde{M}_{\overline{\varepsilon}}\neq\emptyset$ but $\Sigmaw_{\varepsilon_1}^5\cap \partial \widetilde{M}_{\overline{\varepsilon}}=\emptyset$ implies that some translated copy  $\Sigmaw_{\varepsilon_1}^{\overline \tau} $,  $\overline{\tau}\in (0,5)$, must obtain its last point of contact with $\widetilde{M}_{\overline{\varepsilon}}$ at an interior point of $\widetilde{M}_{\overline{\varepsilon}}$. This contradicts the tangency principle, Theorem~\ref{tangency}, and concludes the proof of Theorem~\ref{finiteboundary}.
\end{proof}

\section{The horizontal asymptotic boundary. }\label{nonfinite}
In this section we study the horizontal asymptotic boundary and prove Theorem~\ref{main2}. We begin by proving the following lemma.

\begin{lemma}\label{nonfinitelemma}
    Let $M$ be a properly immersed translator in $\HitR$ and let $\Gamma_\pm$ be a geodesic in $\Hi \times \left\{\pm\infty \right\}$. Assume that there exists $N\in\mathbb N$ such that 
    \[\partial M\subset [\Gamma_\pm\times\mathbb R]\cup [\Hi\times [-N,N]]\text{\, and\, }
     \partial_\infty^vM\subset \partial_\infty[[\Gamma_\pm\times\mathbb R]\cup [\Hi\times [-N,N]].\]
    Let $\mathcal{H}^\pm_1$ and $\mathcal{H}^\pm_2$ be the  half-planes determined by $\Gamma_\pm$. If $\partial^{\pm}_{\infty} M \cap \mathcal{H}^\pm_i \neq \varnothing$, $i=1,2$, then $\partial^{\pm}_{\infty} M \cap \mathcal{H}^\pm_i $, $i=1,2$, is not bounded.
\end{lemma}

\begin{proof} We need to tackle the two cases, $+\infty$ and $-\infty$, separately. However, the following definitions are necessary for both cases.  Arguing by contradiction, assume that $\Delta:=\partial^+_{\infty} M \cap \mathcal {H}^+_1 \neq \emptyset$ (or $\Delta:=\partial^-_{\infty} M \cap \mathcal {H}^-_1 \neq \emptyset$) is bounded. See Figure~\ref{fig:Ab}.

%\begin{comment}
\begin{figure}[htbp!]
\begin{minipage}[]{0.54\linewidth}
\begin{center}
\includegraphics[width=1\linewidth]{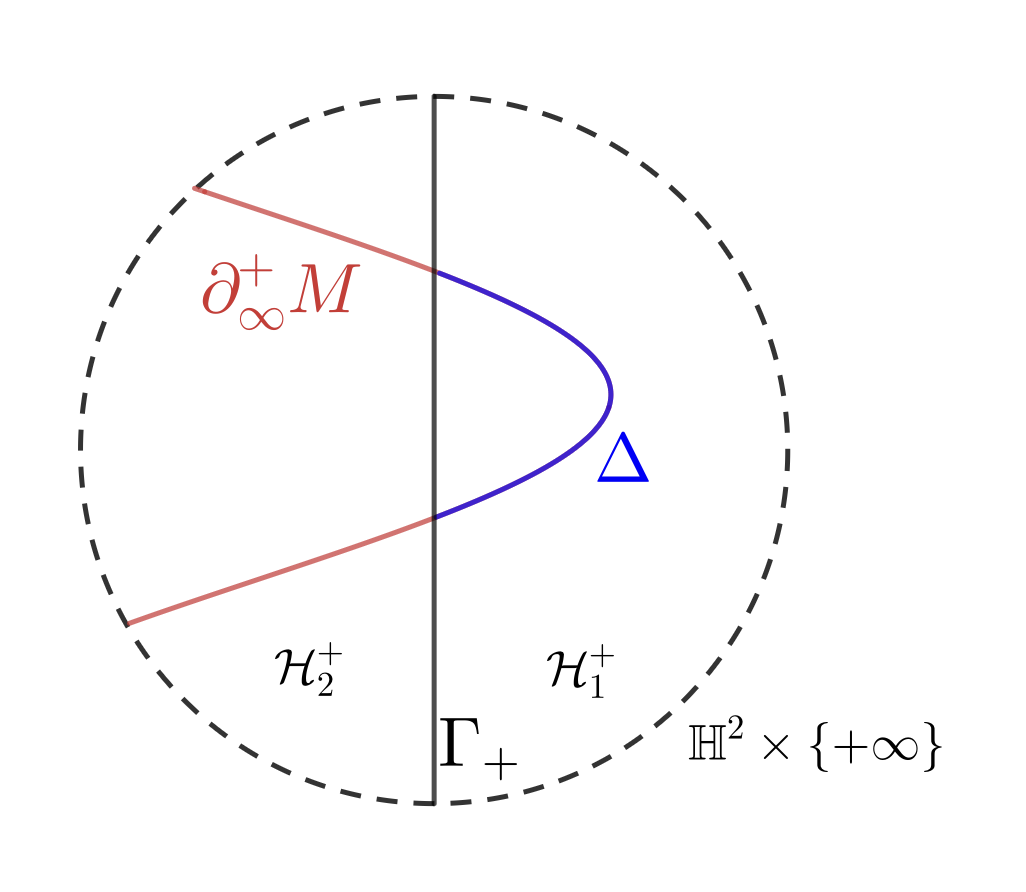}
\end{center}
\end{minipage}
\begin{minipage}[]{0.48\linewidth}
\begin{center}\vspace{0.1cm}
\includegraphics[width=1\linewidth]{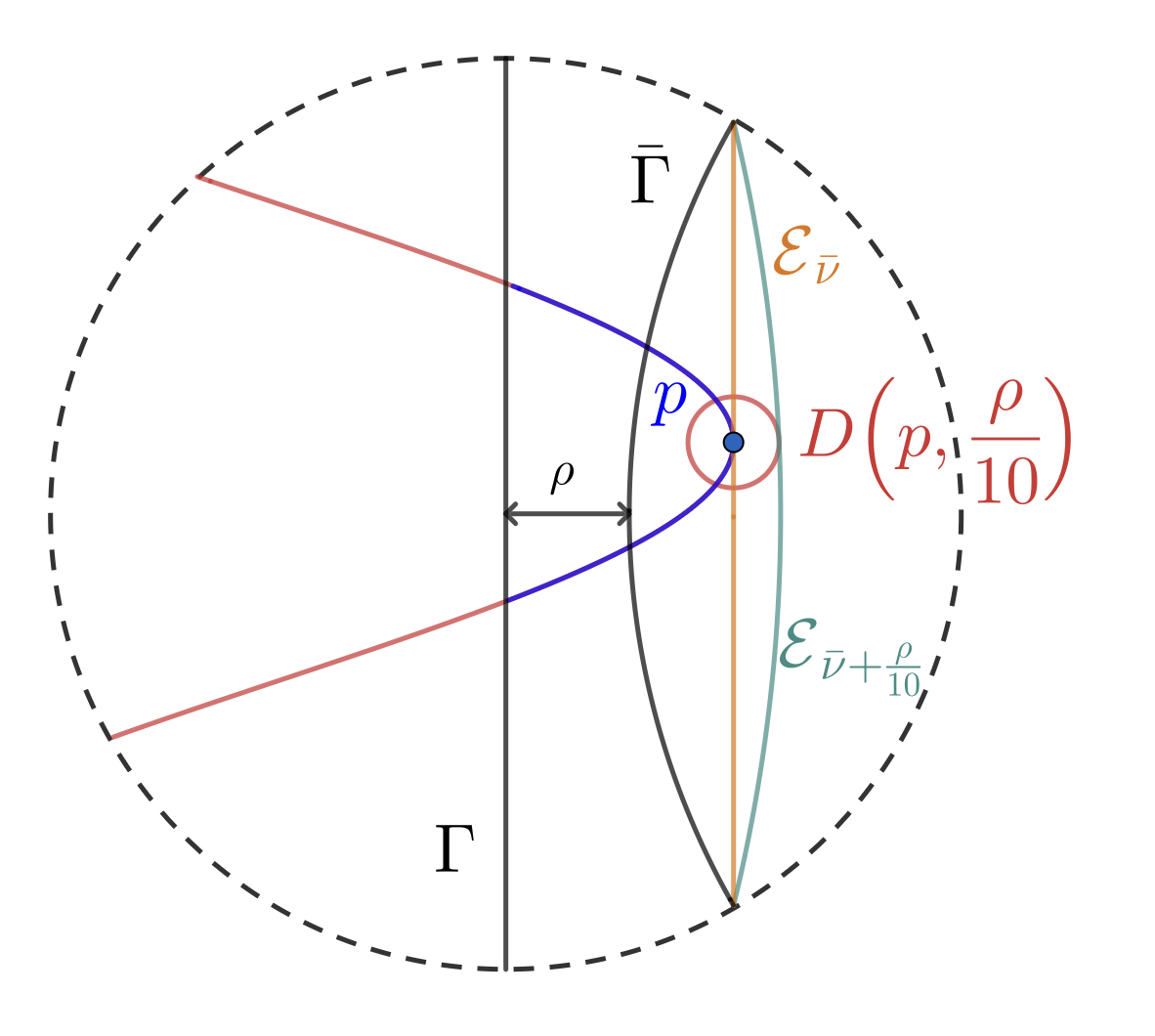}
\end{center}
\end{minipage}
\caption{The constructions in the proof of Lemma \ref{nonfinitelemma}.}
\label{fig:Ab}
\end{figure}
%\end{comment}

Abusing the notations, let $\Delta, \Gamma$ and $\mathcal {H}_1$ denote the projection of $\Delta, \Gamma_\pm$ and $\mathcal {H}^\pm_1$ in $\Hi \times \left\{0\right\}$. Let $\overline\Gamma$  in $\mathcal {H}_1$ be  a geodesic disjoint from $\Gamma$ to be fixed later, let $\mathcal E_\nu$, $\nu>0$, denote the curve that is equidistant from $\overline\Gamma$ so that the distance between $\overline\Gamma$ and $\mathcal E_\nu$ is $\nu$ and $\mathcal E_\nu$ is contained in the component of $[\mathbb{H}^2\times \{0\}]\setminus \overline\Gamma$ that does not contain $\Gamma$. Using this notation, we can fix $\overline\Gamma$ and $\overline{\nu}>0$ such that the following hold:
\begin{itemize}
    \item $\overline\Gamma\cap\Delta\neq \emptyset$;
    \item $\mathcal E_{\overline{\nu}}\cap\Delta\neq \emptyset$;
    \item for any $\nu>\overline{\nu}$,  $\mathcal E_\nu\cap\Delta= \emptyset$.
\end{itemize}
Let $p$ be a point in $\mathcal E_{\overline{\nu}}\cap\Delta$ and let $\overline\Gamma^\perp$ denote the geodesic through $p$ and perpendicular to $\mathcal E_{\overline{\nu}}$. Let $\rho$ denote the distance between $\Gamma$ and $\overline{\Gamma}$. Finally, let $R>0$ such that $\Delta\subset D(\vec 0,R)$, where $D(q,r)$ denote the hyperbolic disk centered at $q\in\Hi$ of radius $r$.

We now begin the proof of the case $\Delta:=\partial^+_{\infty} M \cap \mathcal {H}_1 \neq \emptyset$. Given $\tau>0$, let $M_\tau$ be $M$ translated down by $\tau$ and, given $\sigma>0$, let $M_\tau^\sigma:=M_\tau\cap [\mathbb{H}^2\times [-\sigma,+\infty)]$. Note that for any $\sigma>0$ there exists $\overline{\tau}:=\tau(\sigma)>0$ such that the following holds. \begin{itemize}
    \item for any $\nu\geq \frac{\rho}{10}+\overline{\nu}$, $M_{\overline{\tau}}^\sigma \cap [\mathcal E_\nu\times\mathbb{R}]=\emptyset $;
    \item $M_{\overline{\tau}}^\sigma \cap [D(p,\frac{\rho}{10})\times\{0\}]\neq \emptyset$.
\end{itemize}
See Figure~\ref{fig:Ab}. Let $\Sigmaw\in\Omega$ (see Section \ref{symmetric-translators} for the definition) so that $\Sigmaw$ is tangent to $\mathcal E_{\frac{\rho}{10}+\overline{\nu}}\times\mathbb R$ along 
$ \mathcal E_{\frac{\rho}{10}+\overline{\nu}}\times \{0\} $ and $\partial_\infty^v \Sigmaw\subset \partial_\infty \overline{\Gamma}\times \mathbb{R}$. Fix $\sigma, \overline\tau>0$ so that all points in  $\Sigmaw $    have height greater than $-\frac{\sigma}{2}$. Give $\mu\geq 0$, let  $\Sigmaw^\mu$ denote the hyperbolically translated  $\Sigmaw $ by a distance $\mu$ along the geodesic $\overline\Gamma^\perp$ in the direction of $p$. Then, by construction, the family $\Sigmaw^\mu $, $\mu\in[0,2\frac{\rho}{10}]$, is a continuous family of translators such that \[
\Sigmaw^0 \cap M_{\overline{\tau}}^\sigma=\emptyset \quad \text{and} \quad \Sigmaw^\mu \cap [\partial M_{\overline{\tau}}^\sigma\cup \partial_\infty M_{\overline{\tau}}^\sigma] =\emptyset.\]
However, by construction, since $[D(p,\frac{\rho}{10})\times\{0\}]$ is ``in between'' $\Sigmaw^0 $ and $\Sigmaw^{2\frac{\rho}{10}} $ this family of translators must intersect $M_{\overline{\tau}}^\sigma$ a first time at an interior point. This contradicts the tangency principle, Theorem~\ref{tangency}, and concludes the proof of the case  $\Delta:=\partial^+_{\infty} M \cap \mathcal {H}_1 \neq \emptyset$.

We now prove the case $\Delta:=\partial^-_{\infty} M \cap \mathcal {H}_1 \neq \emptyset$. Choose $\Sigmaw\in\Omega$ as before. 
Recall that given $s>0$ there exists $l(s)>0$ such that the distance between $\partial D(\vec 0,R )\times\mathbb{R}$ and the set $\Sigmaw\cap [\Hi\times \{l(s)\}]$ is greater than $s$.

Given $\tau>0$, let $M_\tau$ be $M$ translated up by $\tau$ and, given $\sigma>0$, let $M_\tau^\sigma:=M_\tau\cap [\mathcal{H}_1\times (-\infty,\sigma]]$. Note that for any $\varepsilon, \sigma>0$ there exists $\overline{\tau}:=\tau(\varepsilon, \sigma)>0$ such that the following holds. 
\begin{itemize}
    \item for any $\nu\geq \frac{\rho}{10}+\overline{\nu}$, $M_{\overline{\tau}}^\sigma \cap [\mathcal E_\nu\times\mathbb{R}]=\emptyset $;
    \item $M_{\overline{\tau}}^\sigma \cap [D(p,\frac{\rho}{10})\times\{0\}]\neq \emptyset$;
       \item for any $r\geq \varepsilon$, $M_{\overline{\tau}}^\sigma \cap [\partial D(\vec 0,R +r )\times \mathbb{R}]=\emptyset $.
\end{itemize}

With these notations in mind, set $s=\rho$, $\overline{l}:=l(\rho)$, $\overline{\varepsilon}:=\varepsilon(\frac\rho{10})$, $\overline{\sigma}:=\sigma(\overline{l})$ and $\overline\tau:=\overline\tau(\overline{\varepsilon},\overline{\sigma})$. Give $\mu\geq 0$, let  $\Sigmaw^\mu$ denote the hyperbolically translated  $\Sigmaw $ by a distance $\mu$ along the geodesic $\overline\Gamma^\perp$ in the direction of $p$.  Then, by construction, the family $\Sigmaw^\mu $, $\mu\in[0,2\frac{\rho}{10}]$, is a continuous family of translators such that \[
\Sigmaw^0 \cap M_{\overline{\tau}}^\sigma=\emptyset \quad \text{and} \quad \Sigmaw^\mu \cap [\partial M_{\overline{\tau}}^\sigma\cup \partial_\infty M_{\overline{\tau}}^\sigma] =\emptyset.\]
However, by construction, since $[D(p,\frac{\rho}{10})\times\{0\}]$ is ``in between'' $\Sigmaw^0 $ and $\Sigmaw^{2\frac{\rho}{10}} $ this family of translators must intersect $M_{\overline{\tau}}^\sigma$ a first time at an interior point. Exactly like in the previous case, this leads to a contradiction and this contradiction finishes the proof of the remaining case, that is when case   $\Delta:=\partial^-_{\infty} M \cap \mathcal {H}_1 \neq \emptyset$.

\end{proof}

 Before we use Lemma~\ref{nonfinitelemma} to prove Theorem~\ref{main2}, we are going to describe some other applications.
As a consequence of Lemma~\ref{nonfinitelemma} we have the following proposition.

\begin{proposition}\label{vertical2}
Let $M$ be a properly immersed translator in $\HitR$ and let $\alpha$ be a closed subset of $\partial_\infty \Hi$ such that $\partial^v_\infty M\subset \alpha\times \mathbb{R}$ and $\partial_\infty^{c,\pm} M\subset \alpha\times \{\pm\infty\}$. Let $\beta$ be a connected component of $\partial_\infty \Hi\setminus \alpha$, let $p,q$ be the end points of $\beta$ and let $\mathcal{H}_{pq}$ be the component of $\Hi\setminus \Gamma_{pq}$ that contains $\beta$, where $\Gamma_{pq}$ denotes complete geodesic in $\Hi$ such that $\partial \Gamma_{pq}=\{p, q\}$. Then
$M\cap[\mathcal{H}_{pq}\times \mathbb{R}]=\emptyset$.
\end{proposition}
\begin{proof}
Since $\partial^v_\infty M\subset \alpha\times \mathbb{R}$ and $\partial_\infty^{c,\pm} M\subset \alpha\times \{\pm\infty\}$, it follows from Lemma~\ref{nonfinitelemma} that $\partial_\infty^\pm M \cap \mathcal{H}_{pq}\times \{\pm\infty\}=\emptyset$.  Note that the region of $\mathcal{H}_{pq}\times\mathbb{R}$ is foliated by vertical geodesic planes. Therefore, if $M\cap[\mathcal{H}_{pq}\times\mathbb{R}]\neq\emptyset$, a standard argument using Theorem~\ref{tangency} would lead to a contradiction. This finishes the proof of the proposition.
\end{proof}

%\begin{comment}
 \begin{figure}[htbp!]
\begin{center}
\includegraphics[width=0.45\linewidth]{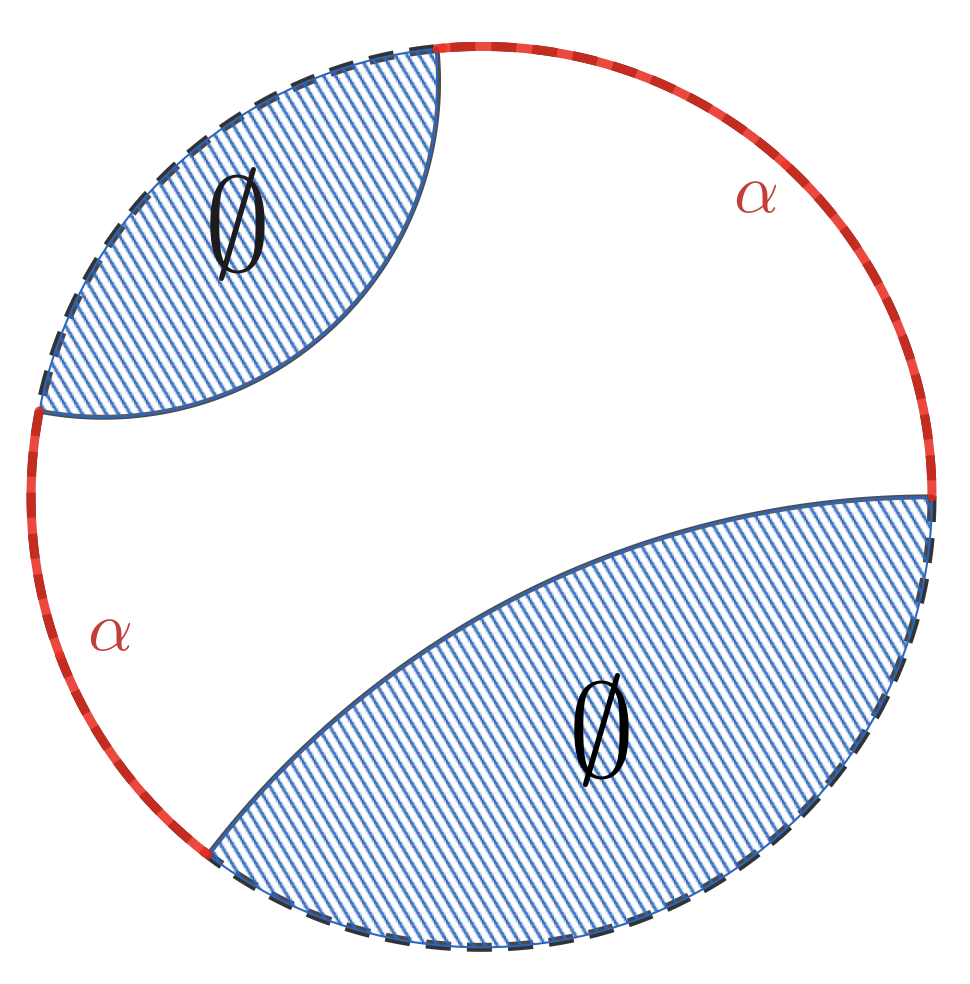}
\end{center}
\caption{A graphical representation of Proposition \ref{vertical2}: in blue the empty regions.}
\label{fig:prop12}
\end{figure}
%\end{comment}

Using Proposition~\ref{vertical2} we can give a different proof of the fact that a translator $M$ cannot be contained in a vertical cylinder over a compact domain or in a horocylinder; see Theorems 17 and 18 in~\cite{lira-martin-2019}, and Theorem 5.9 in \cite{bueno-2018}. To prove that it cannot be contained in a vertical cylinder using Proposition~\ref{vertical2}, observe that in this case $\alpha$ is empty and therefore $M\cap[\mathcal{H}_{pq}\times \mathbb{R}]=\emptyset$ for any $p,q\in \partial_\infty \Hi$. To prove that it cannot be contained in a horocylinder, observe that,  in this case, we have $\alpha=\{r\}$,  where $r$ is the point at infinity of the horocycle. Therefore, $M\cap[\mathcal{H}_{pq}\times \mathbb{R}]=\emptyset$ for any $p,q\in \partial_\infty \Hi$ arbitrarily close to $r\in \partial_\infty \Hi$.

Finally, using Proposition~\ref{vertical2} we can also prove this characterization of vertical geodesic planes.

\begin{proposition}\label{uniqueness}
Let $\Gamma\subset \Hi$ be a complete geodesic and let $\Gamma_1, \Gamma_2\subset\Hi$ be two equidistant curves, equidistant from $\Gamma$. Let $U\subset\Hi$ be the closed region such that $\partial U=\Gamma_1\cup\Gamma_2$. Then, there are two cases. If $\Gamma\not\subset U$ then there is no translator in $U\times \mathbb R$. If $\Gamma\subset U$ then $\Gamma\times \mathbb R$ is the unique translator in $U\times \mathbb R$.
\end{proposition}
\begin{proof}
Let $\partial_\infty \Gamma:=\{p,q\}$ and let $M$ be a translator contained in $U\times \mathbb R$. Observe that $\partial_{\infty}\Gamma_1=\partial_{\infty}\Gamma_2=\partial_{\infty}\Gamma$, hence $\partial^v_\infty M\subset \{p, q\}\times \mathbb{R}$ and $\partial_\infty^c M\subset [\{p\}\times \{\pm\infty\}]\cup [\{q\}\times \{\pm\infty\}]$. We have that $\partial_\infty \Hi\setminus \{p, q\}] $ consists of two open arcs, $\beta_1$ and $\beta_2$. Therefore applying Proposition~\ref{vertical2}, first with $\beta:=\beta_1$ and then with $\beta:=\beta_2$, gives that $M$ must be $\Gamma\times \mathbb{R}$. This finishes the proof of the proposition. 
\end{proof}

%\begin{comment}
 \begin{figure}[htbp!]
\begin{center}
\includegraphics[width=0.85\linewidth]{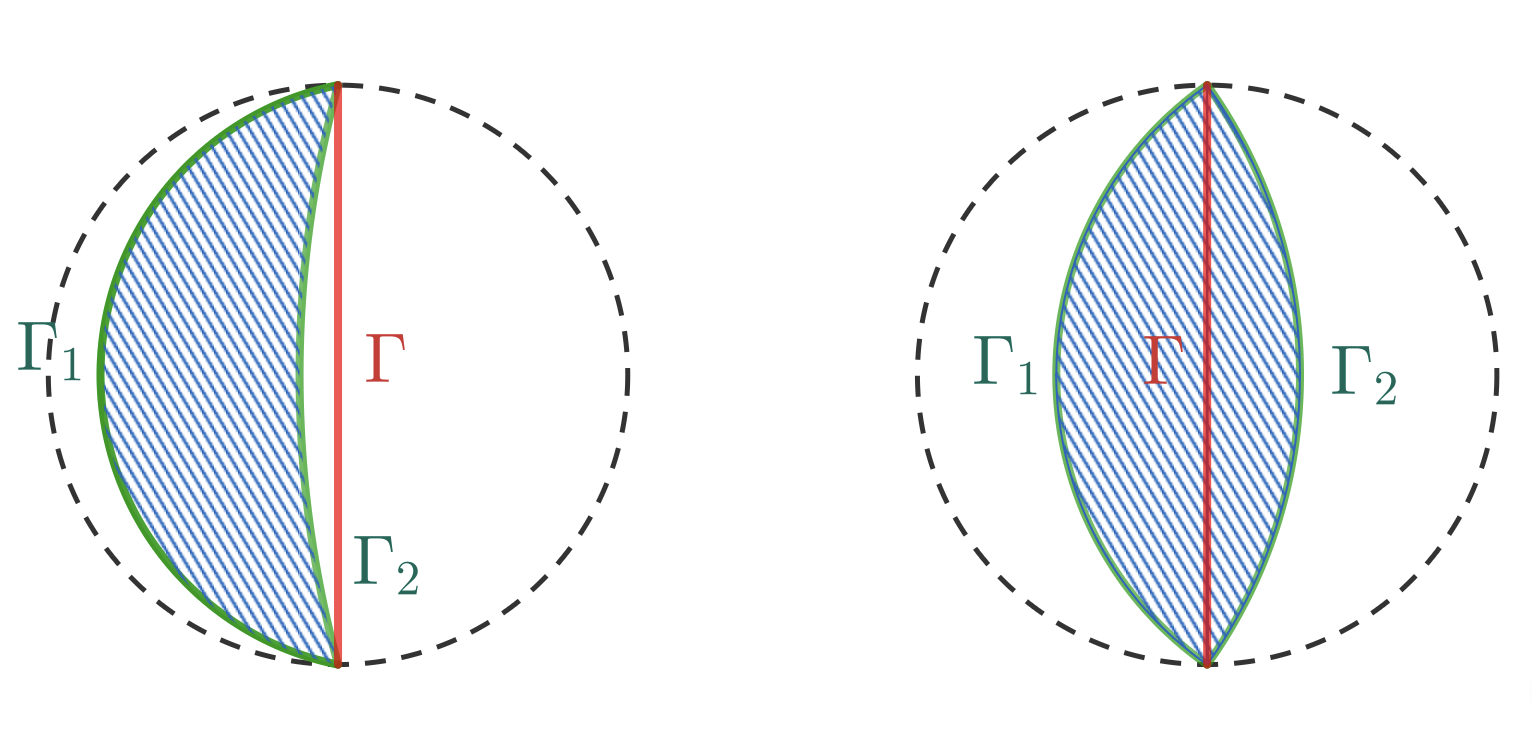}
\end{center}
\caption{A graphical representation of Proposition \ref{uniqueness}. In blue two examples of domains $U\subset\mathbb H^2$ bounded by equidistant lines: on the left there are no translators in $U\times\mathbb R$, on the right $\Gamma\times\mathbb R$ is the unique translator in $U\times\mathbb R$.}
\label{fig:prop13}
\end{figure}
%\end{comment}

 We now use Lemma~\ref{nonfinitelemma} to prove Theorem~\ref{main2}.

\begin{proof}[Proof of Theorem~\ref{main2}]

Let us first consider $\partial_\infty^+ M$. Arguing by contradiction, assume that $\alpha$ is not a geodesic arc. This being the case, there exists a complete geodesic $\Gamma$ and a compact sub-arc $\beta\subset \alpha$ such that $\beta \not\subset \Gamma$ and $\partial \beta\in\Gamma$. See Figure~\ref{fig:Cb}.

%\begin{comment}
 \begin{figure}[htbp!]
\begin{center}
\includegraphics[width=0.6\linewidth]{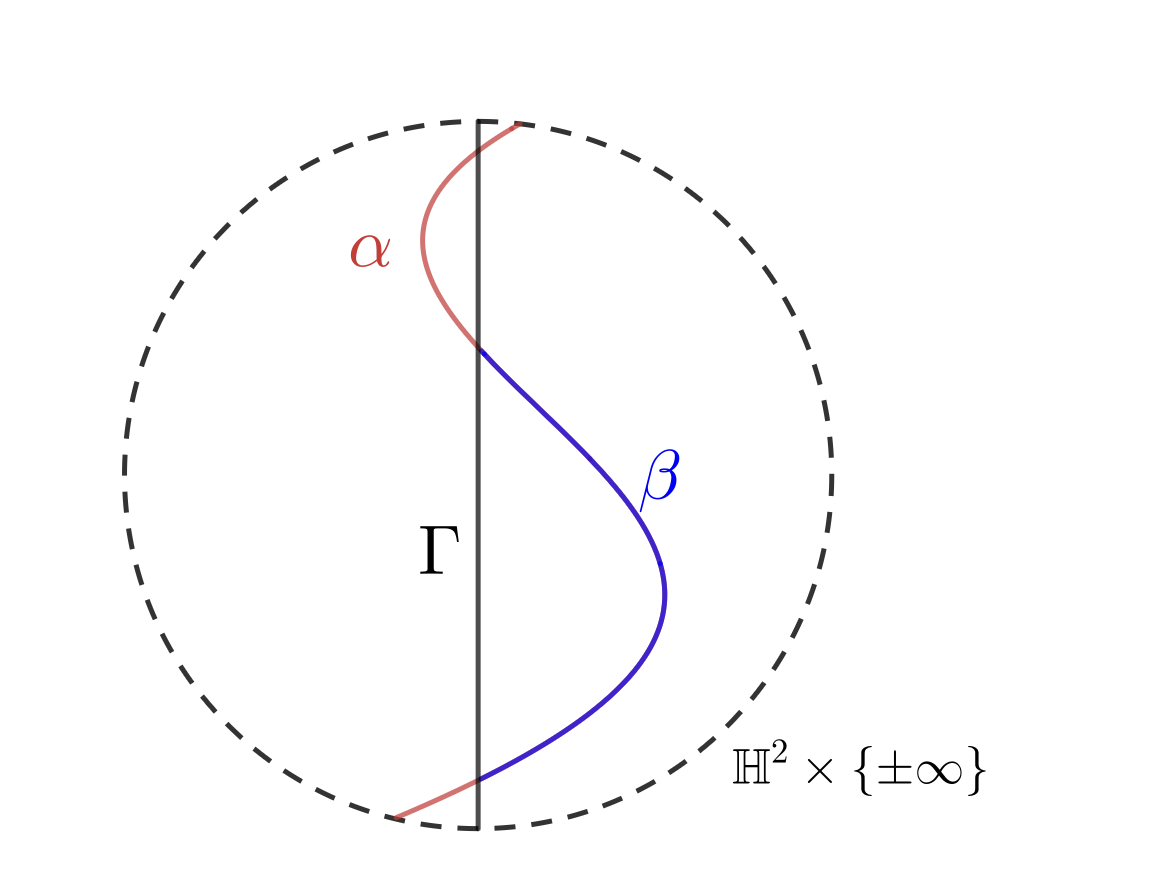}
\end{center}
\caption{The construction in the proof of Theorem \ref{main2}.}
\label{fig:Cb}
\end{figure}
%\end{comment}

By abusing the notation, let $M$ denote a connected component of $M\setminus \Gamma\times \mathbb R$ that has $\beta$ in its non-finite asymptotic boundary. Since $\overline{[\partial_{\infty}^{+} M\setminus \alpha ]}\cap \alpha =\emptyset$ then $\overline{[\partial_{\infty}^{+} M\setminus \alpha ]}\cap \beta=\emptyset$ and after applying a sequence of downward translations, and possibly replacing $\Gamma$ with a nearby geodesic, we can assume that $\beta$ is the only asymptotic boundary component of a connected component of $\Sigma\cap \Hi\times [0,+\infty)$ that has $\beta$ in its  asymptotic boundary and that 
\[\partial \Sigma\subset [\Gamma \times\mathbb R]\cup [\Hi\times {0}\text{\, and\, }
     \partial_\infty^v\Sigma=\emptyset.\]
     
In particular, $\Sigma$ satisfies the hypothesis of~Lemma~\ref{nonfinitelemma}. We can then apply Lemma~\ref{nonfinitelemma} to such component to obtain a contradiction. This finishes the proof for $\partial_\infty^+ M$. The proof for $\partial_\infty^- M$ is exactly the same but it uses upward translations instead of downward translations. 
\end{proof}

The next result and the second bullet in Theorem~\ref{main} are a simple corollary of Theorem~\ref{main2}.

\begin{corollary}\label{nonfinitethm} 
Let $M$ be a properly immersed translator in $\HitR$ with compact boundary and suppose that its positive horizontal asymptotic boundary, $\partial_\infty^+ M$, respectively its negative horizontal asymptotic boundary, $\partial_\infty^+ M$, is a proper family of disjoint immersed continuous curves and let $\alpha$ be a connected component of $\partial_\infty^+ M$, respectively $\partial_\infty^+ M$, so that $M\cup \alpha $ is a continuous surface with boundary. Then $\alpha$ is complete geodesic in $\Hi \times \left\{ + \infty  \right\}$, respectively $\Hi \times \left\{ - \infty  \right\}$.
\end{corollary}

\noindent
{\bf Acknowledgments.} Giuseppe Pipoli was partially supported by INdAM - GNSAGA Project, codice CUP E53C24001950001 and PRIN project 20225J97H5. João Paulo dos Santos was partially supported by CNPq, Brazil, grant number 402589/2022-0, and he thanks the Mathematics Department of King's College London for the hospitality, where part of this work was conducted.

\end{document}